\documentclass{scrartcl}
\usepackage{amsfonts,amsthm,amssymb}
\usepackage{mathtools}
\usepackage[all]{xy} 
\usepackage{url}
\usepackage{tikz}

\newtheorem{thm}{Theorem}[section]
\newtheorem{lem}[thm]{Lemma}

\newtheorem{prop}[thm]{Proposition}

\theoremstyle{definition}
  \newtheorem{defn}[thm]{Definition}

\title{The third symmetric potency of the circle and the Barnette sphere}
\author{Yuki Nakandakari \and Shuichi Tsukuda}

\newcommand{\relmid}{\mathrel{}\middle|\mathrel{}}

\newcommand{\Ex}[1]{E(#1)}
\newcommand{\Ix}[1]{I(#1)}
\newcommand{\spot}[2][k]{(#2)^{(#1)}}
\newcommand{\spotc}[1][3]{(S^1)^{(#1)}}
\newcommand{\card}[1]{|#1|}
\newcommand{\face}[1]{[#1]}

\tikzset{%
 font=\bfseries\upshape,
 face/.style={fill opacity=0.65},
 frame/.style={thin, draw=black},
 thickedge/.style={preaction={draw, line width=4pt, white, opacity=0.8}, line width=3pt},
 ax/.style={->},
}

\begin{document}

\maketitle

\renewcommand{\thefootnote}{}

\begin{abstract}
 We give an elementary (not cut just paste) proof of results of Bott and Shchepin:
the space of non-empty subsets of a circle of cardinality at most 3, 
which is called the third symmetric potency of the circle,
is homeomorphic to a 3-sphere and the inclusion of 
the space of one element subsets is a trefoil knot. 
Moreover, we give an explicit simplicial decomposition of 
the third symmetric potency of the circle which is isomorphic to the Barnette sphere.
\end{abstract}

\section{Introduction}
The space of non-empty finite subsets of a topological space 
of cardinality at most $k$ have been studied in various areas of mathematics 
under various names and notations. 
Since our purpose of this paper is to study 
the case of the circle and $k=3$, which seems to be the origin of 
the study of these spaces initiated by 
K.~Borsuk \cite{borsuk_third_1949} and R.~Bott \cite{bott_third_1952}, 
we use the original name and notation.

\begin{defn}
 Let $X$ be a topological space and $k$ be a positive integer.
 The set of non-empty subsets of $X$ of cardinality at most $k$ 
 equipped with the quotient topology given by the canonical projection 
 from the $k$-fold Cartesian product is called the $k$-th symmetric potency of $X$ 
 and denoted by $\spot{X}$. 
 That is, as a set, 
 \[
  \spot{X}=\left\{ A\subset X \relmid 0<\card{A}\leq k \right\}
 \]
 where $\card{A}$ denotes the cardinality of $A$,  
 and the topology is induced by the projection
 \[
  \pi\colon X^k \to \spot{X}
 \]
 given by $\pi(x_1,\dots,x_k)=\{x_1,\dots, x_k\}$.
\end{defn}

Note that we have canonical inclusions
$\spot[1]{X}\subset \spot[2]{X} \subset \spot[3]{X} \cdots$, 
and $\spot[1]{X}$ is identified with $X$ by the projection $\pi\colon X \to \spot[1]{X}$.
In \cite{bott_third_1952}, Bott identified the topology of $\spotc$ 
and E.~V.~Schepin did the inclusion $S^1=\spotc[1]\subset \spotc$.

\begin{thm}[R.~Bott \cite{bott_third_1952}]
 \label{main1}
 $\spotc$ is homeomorphic to $S^3$.
\end{thm}
\begin{thm}[E.~V.~Schepin \cite{Shchepin}]
 \label{main2}
 The inclusion $S^1=\spotc[1]\subset \spotc\cong S^3$ is a trefoil knot.
\end{thm}

Various proofs of these Theorems are known such as given by 
S.~Kallel and D.~Sjerve \cite{kallel_remarks_2009},
J.~Mostovoy \cite{mostovoy_lattices_2004},
C.~Tuffley \cite{tuffley_finite_2002}.
In this paper, we give an elementary (not cut just paste) proof of 
Theorems~\ref{main1} and \ref{main2}. 
Moreover, we give an explicit simplicial decomposition of $\spotc$ 
which is isomorphic to the Barnette sphere \cite{barnette}.

\section{Pasting}
Fundamental topological properties of symmetric potencies 
are studied by D.~Handel \cite{handel_homotopy_2000}, and
the following can be found as \cite[Proposition 2.7]{handel_homotopy_2000}.
For the sake of completeness, we give a proof
which is essentially the same as that of \cite[Proposition 2.7]{handel_homotopy_2000}.

\begin{prop}
If $X$ is Hausdorff, then so is $\spot{X}$.
\end{prop}
\begin{proof}
 Let $\pi\colon X^k \to \spot{X}$ be the canonical projection.

 For a subset $S\subset X$, we define subsets $\Ex{S}$ and $\Ix{S}$ of $\spot{X}$ as follows:
 \begin{align*}
  \Ex{S}&:=\left\{A\in \spot{X} \relmid A\subset S\right\} \\
  \Ix{S}&:=\left\{A\in \spot{X} \relmid A\cap S \ne \emptyset\right\}
 \end{align*}
 Note that
 $\pi^{-1}(\Ex{S})=S^k$
 and 
 $\Ix{S}^c=\Ex{S^c}$.
 In particular,
 if $S$ is open (resp.\ closed), then so is $\Ex{S}$ and hence so is $\Ix{S}$.
 
 Let $x=\{x_1,\dots,x_l\}, y=\{y_1,\dots,y_m\}\in \spot{X}$ be two distinct points.
 We may assume that $x_1\not\in y$. Since $X$ is Hausdorff, we can find
 open subsets $U,V\subset X$ satisfying
 $x_1\in U$, $y=\{y_1,\dots,y_m\}\subset V$ and $U\cap V = \emptyset$.
 Then $\Ix{U}$ and $\Ex{V}$ are open and 
 $x\in \Ix{U}$, $y \in \Ex{V}$.
 Since $V\subset U^c$, $\Ex{V}\subset \Ex{U^c} = \Ix{U}^c$ whence $\Ix{U}\cap \Ex{V} = \emptyset$.
\end{proof}

Our starting point is Morton's prism. 
Consider the following subspaces of $\mathbb{R}^3$ (see Figure~\ref{figureP}):
\begin{align*}
 P  &=\left\{(x, y, z)\in\mathbb{R}^3 \relmid x\leq y \leq  z\leq  x+1  \text{ and }  0\leq x+y+z\leq 1\right\} \\
 S & =\left\{(x, y, z)\in P \relmid x=y \text{ or } y=z \text{ or } z=x+1 \right\} \\
 D & =\left\{(x, y, z)\in P \relmid x=y=z \text{ or } x=y=z-1 \text{ or } x+1=y=z\right\}
\end{align*}
The vertices of $P$ are the following:
\begin{align*}
 \mathbf{0}&: (0,0,0) & \mathbf{1}&: (-2/3,1/3,1/3) & \mathbf{2}&: (-1/3,-1/3,2/3) \\
 \mathbf{4}&: (0,0,1) & \mathbf{5}&: (1/3,1/3,1/3) & \mathbf{6}&: (-1/3,2/3,2/3) 
\end{align*}

H.~R.~Morton \cite{morton_symmetric_1967} used the prism $P$ to describe symmetric products of the circle: 
the symmetric product $\left(S^1\right)^3/\mathfrak{S}_3$ is 
obtained from $P$ by identifying the bottom face $\face{0,1,2}$ and the 
top face $\face{4,5,6}$, hence homeomorphic to a solid torus.

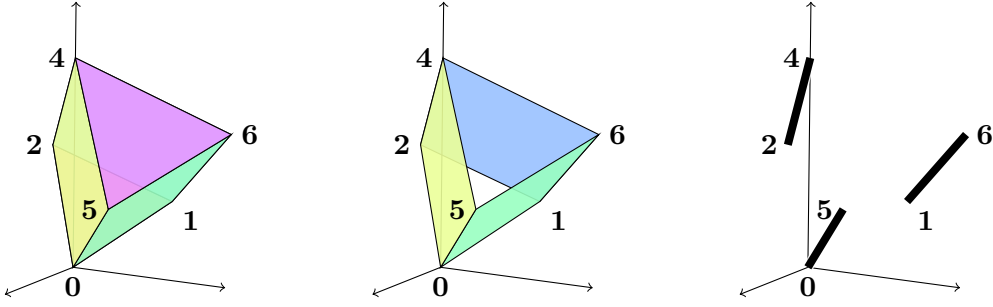
\begin{figure}
\tikzset{face/.append style={fill opacity=0.9}}
\definecolor{mycolor0}{RGB}{255,153,153}
\definecolor{mycolor1}{RGB}{235,255,153}
\definecolor{mycolor2}{RGB}{153,255,193}
\definecolor{mycolor3}{RGB}{153,193,255}
\definecolor{mycolor4}{RGB}{234,153,255}
 \newcommand*{\mylabel}[1]{%
 \ifcase #1 0 \or 1 \or 2 \or 5 \or 6 \or 4 \or \fi}
\begin{minipage}[b]{0.32\hsize}
 \centering
 \begin{tikzpicture}[scale=0.8]
 \coordinate (e0) at (0, 0);
 \coordinate (e1) at (-1.12, -0.45);
 \coordinate (e2) at (2.5, -0.34);
 \coordinate (e3) at (0.05, 4.4);
 \draw[ax] (e0) -- (e1);
 \draw[ax] (e0) -- (e2);
 \coordinate (0) at (0, 0);
 \coordinate (1) at (1.63, 1.09);
 \coordinate (2) at (-0.33, 2.03);
 \coordinate (3) at (0.58, 0.96);
 \coordinate (4) at (2.6, 2.2);
 \coordinate (5) at (0.04, 3.47);
 \node at (1) [below right] {\mylabel{1}};
 \node at (2) [left]{\mylabel{2}};
 \draw[frame, face, fill=mycolor0] (0)--(1)--(2)--cycle;
 \draw[frame, face, fill=mycolor3] (1)--(2)--(5)--(4)--cycle;
 \draw[ax] (e0) -- (e3);
 \node at (0) [below] {\mylabel{0}};
 \draw[frame, face, fill=mycolor2] (0)--(1)--(4)--(3)--cycle;
 \draw[frame, face, fill=mycolor1] (0)--(3)--(5)--(2)--cycle;
 \node at (4) [right] {\mylabel{4}};
 \node at (5) [left] {\mylabel{5}};
 \draw[frame, face, fill=mycolor4] (3)--(4)--(5)--cycle;
 \node at (3) [left] {\mylabel{3}};
 \end{tikzpicture}
\end{minipage}
\begin{minipage}[b]{0.32\hsize}
 \centering
 \begin{tikzpicture}[scale=0.8]
 \coordinate (e0) at (0, 0);
 \coordinate (e1) at (-1.12, -0.45);
 \coordinate (e2) at (2.5, -0.34);
 \coordinate (e3) at (0.05, 4.4);
 \draw[ax] (e0) -- (e1);
 \draw[ax] (e0) -- (e2);
 \coordinate (0) at (0, 0);
 \coordinate (1) at (1.63, 1.09);
 \coordinate (2) at (-0.33, 2.03);
 \coordinate (3) at (0.58, 0.96);
 \coordinate (4) at (2.6, 2.2);
 \coordinate (5) at (0.04, 3.47);
 \node at (1) [below right] {\mylabel{1}};
 \node at (2) [left]{\mylabel{2}};
 \draw[frame, face, fill=mycolor3] (1)--(2)--(5)--(4)--cycle;
 \draw[ax] (e0) -- (e3);
 \node at (0) [below] {\mylabel{0}};
 \draw[frame, face, fill=mycolor2] (0)--(1)--(4)--(3)--cycle;
 \draw[frame, face, fill=mycolor1] (0)--(3)--(5)--(2)--cycle;
 \node at (4) [right] {\mylabel{4}};
 \node at (5) [left] {\mylabel{5}};
 \node at (3) [left] {\mylabel{3}};
 \end{tikzpicture}
\end{minipage}
\begin{minipage}[b]{0.32\hsize}
 \centering
 \begin{tikzpicture}[scale=0.8]
 \coordinate (e0) at (0, 0);
 \coordinate (e1) at (-1.12, -0.45);
 \coordinate (e2) at (2.5, -0.34);
 \coordinate (e3) at (0.05, 4.4);
 \draw[ax] (e0) -- (e1);
 \draw[ax] (e0) -- (e2);
 \draw[ax] (e0) -- (e3);
 \coordinate (0) at (0, 0);
 \coordinate (1) at (1.63, 1.09);
 \coordinate (2) at (-0.33, 2.03);
 \coordinate (3) at (0.58, 0.96);
 \coordinate (4) at (2.6, 2.2);
 \coordinate (5) at (0.04, 3.47);
 \draw[thickedge] (0) -- (3);
 \draw[thickedge] (1) -- (4);
 \draw[thickedge] (2) -- (5);
 \node at (1) [below right] {\mylabel{1}};
 \node at (2) [left]{\mylabel{2}};
 \node at (0) [below] {\mylabel{0}};
 \node at (4) [right] {\mylabel{4}};
 \node at (5) [left] {\mylabel{5}};
 \node at (3) [left] {\mylabel{3}};
 \end{tikzpicture}
\end{minipage}
\caption{$P,S$ and $D$}
\label{figureP}
\end{figure}

Let $p\colon P \to \spotc$ be the map defined by 
$p(x,y,z)=\{e^{2\pi ix},e^{2\pi iy},e^{2\pi iz}\}$.
Define a equivalence relation on $P$ by 
$(x,y,z)\sim (x',y',z') \Leftrightarrow p(x,y,z)=p(x',y',z')$
and give $P/{\sim}$ the quotient topology.
Let $\bar{p}$ be the induced map:
\[
\bar{p}\colon P/{\sim} \rightarrow \spotc
\]

\begin{prop}
The map $\bar{p}\colon P/{\sim} \rightarrow \spotc$ is a homeomorphism. 
Moreover, restrictions of $\bar{p}$ give homeomorphisms 
$S/{\sim}  \rightarrow \spotc[2]$ and 
$D/{\sim} \rightarrow \spotc[1]$:
\[
\xymatrix{
& D/{\sim} \ar@{}[r]|*{\subset} \ar[d]^{\bar{p}}_{\cong} & S/{\sim} \ar@{}[r]|*{\subset} \ar[d]^{\bar{p}}_{\cong} & 
P/{\sim} \ar[d]^{\bar{p}}_{\cong} \\
S^1 \ar@{=}[r] & \spotc[1] \ar@{}[r]|*{\subset} & \spotc[2] \ar@{}[r]|*{\subset} & \spotc 
}
\]
\end{prop}
\begin{proof}
 It is easy to see that the map $p\colon P \to \spotc$ is surjective. 
 Since $P$ is compact and $\spotc$ is Hausdorff, the map $\bar{p}$ is a homeomorphism. 
 It is straightforward to see that 
 $p^{-1}\left(\spotc[2]\right)=S$, $p^{-1}\left(\spotc[1]\right)=D$, and the rest of the assertions follow.
\end{proof}

It is straightforward to see the following.
\begin{lem}
 \label{identifytopbottom}
 The equivalence relation $\sim$ is given by affinely identifying
 the following triangles \textup{(}$2$-simplices\textup{)} of the boundary of the prism $P$:

\begin{minipage}[c]{.47\hsize}
 \begin{align*}
  & \text{$\face{0,1,2}$ and $\face{4,5,6}$}\\ 
  & \text{$\face{0,1,6}$ and $\face{4,1,6}$} \\
  & \text{$\face{1,2,4}$ and $\face{5,2,4}$} \\
  & \text{$\face{2,0,5}$ and $\face{6,0,5}$} 
 \end{align*}
\end{minipage}
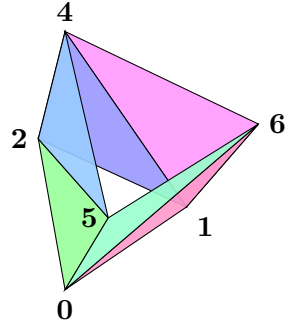
\begin{minipage}[c]{.47\hsize}
 \centering
 \begin{tikzpicture}[scale=0.4]
 \clip (-2.4, -1) rectangle (8, 11);
  \tikzset{face/.append style={fill opacity=0.9}}
 \newcommand*{\mylabel}[1]{%
 \ifcase #1 0 \or 1 \or 2 \or  \or 5 \or 6 \or 4 \or  \or  \or  \or  \or \fi}
\definecolor{mycolor0}{RGB}{255,153,153}
\definecolor{mycolor1}{RGB}{255,204,153}
\definecolor{mycolor2}{RGB}{255,255,153}
\definecolor{mycolor3}{RGB}{255,204,153}
\definecolor{mycolor4}{RGB}{255,255,153}
\definecolor{mycolor5}{RGB}{255,153,153}
\definecolor{mycolor6}{RGB}{153,255,153}
\definecolor{mycolor7}{RGB}{153,255,153}
\definecolor{mycolor8}{RGB}{153,255,204}
\definecolor{mycolor9}{RGB}{153,255,204}
\definecolor{mycolor10}{RGB}{153,204,255}
\definecolor{mycolor11}{RGB}{153,204,255}
\definecolor{mycolor12}{RGB}{153,153,255}
\definecolor{mycolor13}{RGB}{153,153,255}
\definecolor{mycolor14}{RGB}{255,153,255}
\definecolor{mycolor15}{RGB}{255,153,255}
\definecolor{mycolor16}{RGB}{255,153,204}
\definecolor{mycolor17}{RGB}{255,153,204}
 \coordinate (0) at (0, 0);
 \coordinate (1) at (3.99, 2.74);
 \coordinate (2) at (-0.88, 5);
 \coordinate (3) at (1.11, 2.7);
 \coordinate (4) at (1.42, 2.38);
 \coordinate (5) at (6.36, 5.49);
 \coordinate (6) at (0, 8.56);
 \coordinate (7) at (2.71, 5.67);
 \coordinate (8) at (0.59, 0.99);
 \coordinate (9) at (5.01, 3.93);
 \coordinate (10) at (-0.5, 6.54);
 \node at (1) [below right] {\mylabel{1}};
 \node at (2) [left] {\mylabel{2}};
 \draw[frame, face, fill=mycolor14] (6)--(1)--(5)--cycle;
 \draw[frame, face, fill=mycolor13] (2)--(1)--(6)--cycle;
 \node at (0) [below] {\mylabel{0}};
 \draw[frame, face, fill=mycolor17] (1)--(0)--(5)--cycle;
 \node at (5) [right] {\mylabel{5}};
 \draw[frame, face, fill=mycolor7] (2)--(4)--(0)--cycle;
 \node at (6) [above] {\mylabel{6}};
 \draw[frame, face, fill=mycolor11] (4)--(2)--(6)--cycle;
 \draw[frame, face, fill=mycolor9] (5)--(0)--(4)--cycle;
 \node at (4) [left] {\mylabel{4}};
 \end{tikzpicture}
\end{minipage}
\end{lem}

These identifications 
can be visually seen as follows.

We first identify the sides. 
By a half twist clockwise of  the top of the prism, we get an octahedron of Figure \ref{Foct} 
and expand the top and the bottom of the octahedron to get a parallelepiped of Figure \ref{F4}.  
Further twisting and pushing down the top to 
identify the sides,
we get the triangular bipyramid of Figure \ref{Fbip}.
Figure \ref{5} shows $S/{\sim}$ in the triangular bipyramid. 
We see that 
$S/{\sim}$ is a full and a half twisted M\"{o}bius band and $D/{\sim}$ is its boundary.
In particular, ${D/{\sim}}$ is a trefoil knot in the triangular bipyramid.

Now, $P/{\sim}$ is obtained by identifying the surface of 
the upper and the lower pyramids of the triangular bipyramid,  
we see that $P/{\sim}$ is homeomorphic to $S^3$.

You may find a JavaScript animation of these identifications at \cite{STU}.

In the next section, we give a more explicit homeomorphism 
from $P/{\sim}$ to $S^3$ as a simplicial map.


\begin{figure}[ht]
\newcommand*{\mylabel}[1]{%
\ifcase #1 0 \or 1 \or 2 \or  \or 5 \or 6 \or 4 \or  \or  \or  \or  \or \fi}
\definecolor{mycolor0}{RGB}{255,153,153}
\definecolor{mycolor1}{RGB}{255,204,153}
\definecolor{mycolor2}{RGB}{255,255,153}
\definecolor{mycolor3}{RGB}{255,204,153}
\definecolor{mycolor4}{RGB}{255,255,153}
\definecolor{mycolor5}{RGB}{255,153,153}
\definecolor{mycolor6}{RGB}{153,255,153}
\definecolor{mycolor7}{RGB}{153,255,153}
\definecolor{mycolor8}{RGB}{153,255,204}
\definecolor{mycolor9}{RGB}{153,255,204}
\definecolor{mycolor10}{RGB}{153,204,255}
\definecolor{mycolor11}{RGB}{153,204,255}
\definecolor{mycolor12}{RGB}{153,153,255}
\definecolor{mycolor13}{RGB}{153,153,255}
\definecolor{mycolor14}{RGB}{255,153,255}
\definecolor{mycolor15}{RGB}{255,153,255}
\definecolor{mycolor16}{RGB}{255,153,204}
\definecolor{mycolor17}{RGB}{255,153,204}
\begin{tabular}{cc}
\begin{minipage}[b]{0.47\hsize}
\centering
\begin{tikzpicture}[scale=1.7]
\coordinate (0) at (0, 0);
\coordinate (1) at (0.86, 0.51);
\coordinate (2) at (-0.59, 0.81);
\coordinate (3) at (0.1, 0.47);
\coordinate (4) at (0.25, 1.21);
\coordinate (5) at (1.15, 1.53);
\coordinate (6) at (-0.45, 1.86);
\coordinate (7) at (0.32, 1.55);
\coordinate (8) at (0.12, 0.57);
\coordinate (9) at (1, 1);
\coordinate (10) at (-0.52, 1.31);
\newcommand*{\myentry}[1]{%
\ifcase #1 
\draw[frame, face, fill=mycolor0] (0)--(1)--(3)--cycle \or 
\draw[frame, face, fill=mycolor1] (1)--(2)--(3)--cycle \or 
\draw[frame, face, fill=mycolor2] (2)--(0)--(3)--cycle \or 
\draw[frame, face, fill=mycolor3] (4)--(5)--(7)--cycle \or 
\draw[frame, face, fill=mycolor4] (5)--(6)--(7)--cycle \or 
\draw[frame, face, fill=mycolor5] (6)--(4)--(7)--cycle \or 
\draw[frame, face, fill=mycolor6] (2)--(8)--(0)--cycle \or 
\draw[frame, face, fill=mycolor7] (2)--(4)--(8)--cycle \or 
\draw[frame, face, fill=mycolor8] (5)--(0)--(8)--cycle \or 
\draw[frame, face, fill=mycolor9] (5)--(8)--(4)--cycle \or 
\draw[frame, face, fill=mycolor10] (4)--(2)--(10)--cycle \or 
\draw[frame, face, fill=mycolor11] (4)--(10)--(6)--cycle \or 
\draw[frame, face, fill=mycolor12] (2)--(10)--(1)--cycle \or 
\draw[frame, face, fill=mycolor13] (10)--(1)--(6)--cycle \or 
\draw[frame, face, fill=mycolor14] (6)--(1)--(9)--cycle \or 
\draw[frame, face, fill=mycolor15] (6)--(9)--(5)--cycle \or 
\draw[frame, face, fill=mycolor16] (1)--(9)--(0)--cycle \or 
\draw[frame, face, fill=mycolor17] (9)--(0)--(5)--cycle \or 
\draw[thickedge] (1) -- (9) \or 
\draw[thickedge] (5) -- (9) \or 
\draw[thickedge] (2) -- (10) \or 
\draw[thickedge] (6) -- (10) \or 
\draw[thickedge] (0) -- (8) \or 
\draw[thickedge] (4) -- (8) \or 
\node at (0) [below] {\mylabel{0}} \or 
\node at (1) [right] {\mylabel{1}} \or 
\node at (2) [left] {\mylabel{2}} \or 
\node at (3) {\mylabel{3}} \or 
\node at (4) [below right] {\mylabel{4}} \or 
\node at (5) [right] {\mylabel{5}} \or 
\node at (6) [left] {\mylabel{6}} \or 
\node at (7) {\mylabel{7}} \or 
\node at (8) {\mylabel{8}} \or 
\node at (9) {\mylabel{9}} \or 
\node at (10) {\mylabel{10}} \or 
\fi}
\foreach \x in {
27  
,12  
,13  
,14  
,15  
,1  
,2  
,0  
,17  
,16  
,4  
,5  
,3  
,10  
,11  
,6  
,7  
,8  
,9  
,21  
,20  
,19  
,18  
,22  
,23  
,34  
,33  
,25  
,26  
,30  
,29  
,24  
,31  
,32  
,28  
}
{
\myentry{\x};
}
\draw[line width=3pt, rotate=-9, white] (7)+(-0.1,0.2) arc (-230:50:.2 and .1);
\draw[<-, thick, rotate=-9] (7)+(-0.1,0.2) arc (-230:50:.2 and .1);
\end{tikzpicture}
\caption{Prism}
\end{minipage}
& 
\begin{minipage}[b]{0.47\hsize}
\centering
\begin{tikzpicture}[scale=1.7]
\coordinate (0) at (0, 0);
\coordinate (1) at (0.86, 0.51);
\coordinate (2) at (-0.59, 0.81);
\coordinate (3) at (0.1, 0.47);
\coordinate (4) at (0.37, 1.81);
\coordinate (5) at (-0.64, 1.58);
\coordinate (6) at (1.21, 1.2);
\coordinate (7) at (0.32, 1.55);
\coordinate (8) at (-0.61, 1.15);
\coordinate (9) at (0.59, 0.59);
\coordinate (10) at (0.62, 1.15);
\newcommand*{\myentry}[1]{%
\ifcase #1 
\draw[frame, face, fill=mycolor0] (0)--(1)--(3)--cycle \or 
\draw[frame, face, fill=mycolor1] (1)--(2)--(3)--cycle \or 
\draw[frame, face, fill=mycolor2] (2)--(0)--(3)--cycle \or 
\draw[frame, face, fill=mycolor3] (4)--(5)--(7)--cycle \or 
\draw[frame, face, fill=mycolor4] (5)--(6)--(7)--cycle \or 
\draw[frame, face, fill=mycolor5] (6)--(4)--(7)--cycle \or 
\draw[frame, face, fill=mycolor6] (2)--(8)--(0)--cycle \or 
\draw[frame, face, fill=mycolor7] (2)--(4)--(8)--cycle \or 
\draw[frame, face, fill=mycolor8] (5)--(0)--(8)--cycle \or 
\draw[frame, face, fill=mycolor9] (5)--(8)--(4)--cycle \or 
\draw[frame, face, fill=mycolor10] (4)--(2)--(10)--cycle \or 
\draw[frame, face, fill=mycolor11] (4)--(10)--(6)--cycle \or 
\draw[frame, face, fill=mycolor12] (2)--(10)--(1)--cycle \or 
\draw[frame, face, fill=mycolor13] (10)--(1)--(6)--cycle \or 
\draw[frame, face, fill=mycolor14] (6)--(1)--(9)--cycle \or 
\draw[frame, face, fill=mycolor15] (6)--(9)--(5)--cycle \or 
\draw[frame, face, fill=mycolor16] (1)--(9)--(0)--cycle \or 
\draw[frame, face, fill=mycolor17] (9)--(0)--(5)--cycle \or 
\draw[thickedge] (1) -- (9) \or 
\draw[thickedge] (5) -- (9) \or 
\draw[thickedge] (2) -- (10) \or 
\draw[thickedge] (6) -- (10) \or 
\draw[thickedge] (0) -- (8) \or 
\draw[thickedge] (4) -- (8) \or 
\node at (0) [below] {\mylabel{0}} \or 
\node at (1) [right] {\mylabel{1}} \or 
\node at (2) [left] {\mylabel{2}} \or 
\node at (3) {\mylabel{3}} \or 
\node at (4) [above] {\mylabel{4}} \or 
\node at (5) [left] {\mylabel{5}} \or 
\node at (6) [right] {\mylabel{6}} \or 
\node at (7) {\mylabel{7}} \or 
\node at (8) {\mylabel{8}} \or 
\node at (9) {\mylabel{9}} \or 
\node at (10) {\mylabel{10}} \or 
\fi}
\foreach \x in {
27  
,26  
,25  
,20  
,21  
,12  
,10  
,34  
,28  
,1  
,2  
,0  
,13  
,23  
,11  
,7  
,9  
,3  
,5  
,4  
,6  
,8  
,16  
,14  
,17  
,15  
,22  
,18  
,19  
,32  
,24  
,33  
,31  
,29  
,30  
}
{
\myentry{\x};
}
\end{tikzpicture}
\caption{Octahedron}
\label{Foct}
 \end{minipage}
\\
\begin{minipage}[b]{0.47\hsize}
\centering
\begin{tikzpicture}[scale=1.7]
\coordinate (0) at (0, 0);
\coordinate (1) at (0.86, 0.51);
\coordinate (2) at (-0.59, 0.81);
\coordinate (3) at (-0.08, -0.41);
\coordinate (4) at (0.37, 1.81);
\coordinate (5) at (-0.64, 1.58);
\coordinate (6) at (1.21, 1.2);
\coordinate (7) at (0.6, 2.91);
\coordinate (8) at (-0.61, 1.15);
\coordinate (9) at (0.59, 0.59);
\coordinate (10) at (0.62, 1.15);
\newcommand*{\myentry}[1]{%
\ifcase #1 
\draw[frame, face, fill=mycolor0] (0)--(1)--(3)--cycle \or 
\draw[frame, face, fill=mycolor1] (1)--(2)--(3)--cycle \or 
\draw[frame, face, fill=mycolor2] (2)--(0)--(3)--cycle \or 
\draw[frame, face, fill=mycolor3] (4)--(5)--(7)--cycle \or 
\draw[frame, face, fill=mycolor4] (5)--(6)--(7)--cycle \or 
\draw[frame, face, fill=mycolor5] (6)--(4)--(7)--cycle \or 
\draw[frame, face, fill=mycolor6] (2)--(8)--(0)--cycle \or 
\draw[frame, face, fill=mycolor7] (2)--(4)--(8)--cycle \or 
\draw[frame, face, fill=mycolor8] (5)--(0)--(8)--cycle \or 
\draw[frame, face, fill=mycolor9] (5)--(8)--(4)--cycle \or 
\draw[frame, face, fill=mycolor10] (4)--(2)--(10)--cycle \or 
\draw[frame, face, fill=mycolor11] (4)--(10)--(6)--cycle \or 
\draw[frame, face, fill=mycolor12] (2)--(10)--(1)--cycle \or 
\draw[frame, face, fill=mycolor13] (10)--(1)--(6)--cycle \or 
\draw[frame, face, fill=mycolor14] (6)--(1)--(9)--cycle \or 
\draw[frame, face, fill=mycolor15] (6)--(9)--(5)--cycle \or 
\draw[frame, face, fill=mycolor16] (1)--(9)--(0)--cycle \or 
\draw[frame, face, fill=mycolor17] (9)--(0)--(5)--cycle \or 
\draw[thickedge] (1) -- (9) \or 
\draw[thickedge] (5) -- (9) \or 
\draw[thickedge] (2) -- (10) \or 
\draw[thickedge] (6) -- (10) \or 
\draw[thickedge] (0) -- (8) \or 
\draw[thickedge] (4) -- (8) \or 
\node at (0) [below] {\mylabel{0}} \or 
\node at (1) [right] {\mylabel{1}} \or 
\node at (2) [left] {\mylabel{2}} \or 
\node at (3) [below] {\mylabel{3}} \or 
\node at (4) [above] {\mylabel{4}} \or 
\node at (5) [left] {\mylabel{5}} \or 
\node at (6) [right] {\mylabel{6}} \or 
\node at (7) [above] {\mylabel{7}} \or 
\node at (8) [left] {\mylabel{8}} \or 
\node at (9) [below] {\mylabel{9}} \or 
\node at (10) [below] {\mylabel{10}} \or 
\fi}
\foreach \x in {
27  
,26  
,25  
,20  
,21  
,12  
,10  
,34  
,28  
,1  
,2  
,0  
,13  
,23  
,11  
,7  
,9  
,3  
,5  
,4  
,6  
,8  
,16  
,14  
,17  
,15  
,22  
,18  
,19  
,32  
,24  
,33  
,31  
,29  
,30  
}
{
\myentry{\x};
}
\end{tikzpicture}
\caption{Parallelepiped}
\label{F4}
\end{minipage}
&
\begin{minipage}[b]{0.47\hsize}
\centering
\begin{tikzpicture}[scale=0.84]
\coordinate (0) at (0, 0);
\coordinate (1) at (3.16, 0.53);
\coordinate (2) at (-1.18, 0.68);
\coordinate (3) at (0.62, -2.8);
\coordinate (4) at (3.16, 0.53);
\coordinate (5) at (-1.18, 0.68);
\coordinate (6) at (0, 0);
\coordinate (7) at (0.7, 3.96);
\coordinate (8) at (0.43, 1.65);
\coordinate (9) at (0.12, 1.58);
\coordinate (10) at (1.48, 1.6);
\draw[frame, face, fill=mycolor3] (4)--(5)--(7)--cycle;
\draw[frame, face, fill=mycolor1] (1)--(2)--(3)--cycle;
\node at (2) [below left] {\mylabel{2}};
\node at (5) [above left] {\mylabel{5}};
\draw[frame, face, fill=mycolor9] (5)--(8)--(4)--cycle;
\draw[thickedge] (4) -- (8);
\node at (1) [below right] {\mylabel{1}};
\node at (4) [above right] {\mylabel{4}};
\draw[frame, face, fill=mycolor10] (4)--(2)--(10)--cycle;
\draw[thickedge] (2) -- (10);
\draw[frame, face, fill=mycolor8] (5)--(0)--(8)--cycle;
\draw[thickedge] (0) -- (8);
\draw[frame, face, fill=mycolor15] (6)--(9)--(5)--cycle;
\draw[frame, face, fill=mycolor14] (6)--(1)--(9)--cycle;
\draw[thickedge] (5) -- (9);
\draw[thickedge] (1) -- (9);
\draw[frame, face, fill=mycolor11] (4)--(10)--(6)--cycle;
\draw[thickedge] (6) -- (10);
\draw[frame, face, fill=mycolor2] (2)--(0)--(3)--cycle;
\draw[frame, face, fill=mycolor0] (0)--(1)--(3)--cycle;
 \draw[frame, face, fill=mycolor4] (5)--(6)--(7)--cycle;
 \draw[frame, face, fill=mycolor5] (6)--(4)--(7)--cycle;
\node at (3) {\mylabel{3}};
\node at (7) {\mylabel{7}};
\node at (0) [below left] {\mylabel{0}};
\node at (6) [above left] {\mylabel{6}};
\end{tikzpicture}
\caption{Triangular bipyramid}
\label{Fbip}
\end{minipage}
\end{tabular}
\end{figure}

\begin{figure}[ht]
\centering
\begin{tikzpicture}[scale=0.84]
\newcommand*{\mylabel}[1]{%
\ifcase #1 0 \or 1 \or 2 \or  \or 5 \or 6 \or 4 \or  \or  \or  \or  \or \fi}
\definecolor{mycolor0}{RGB}{255, 126, 126}
\definecolor{mycolor1}{RGB}{255, 200, 126}
\definecolor{mycolor2}{RGB}{254, 255, 126}
\definecolor{mycolor3}{RGB}{255, 200, 126}
\definecolor{mycolor4}{RGB}{254, 255, 126}
\definecolor{mycolor5}{RGB}{255, 126, 126}
\definecolor{mycolor6}{RGB}{126, 255, 126}
\definecolor{mycolor7}{RGB}{126, 255, 126}
\definecolor{mycolor8}{RGB}{126, 255, 200}
\definecolor{mycolor9}{RGB}{126, 255, 200}
\definecolor{mycolor10}{RGB}{126, 200, 255}
\definecolor{mycolor11}{RGB}{126, 200, 255}
\definecolor{mycolor12}{RGB}{126, 126, 255}
\definecolor{mycolor13}{RGB}{126, 126, 255}
\definecolor{mycolor14}{RGB}{255, 126, 254}
\definecolor{mycolor15}{RGB}{255, 126, 254}
\definecolor{mycolor16}{RGB}{255, 126, 200}
\definecolor{mycolor17}{RGB}{255, 126, 200}
\coordinate (0) at (0, 0);
\coordinate (1) at (-4.03, 0.05);
\coordinate (2) at (-1.89, -3.49);
\coordinate (3) at (-1.79, -0.98);
\coordinate (4) at (-4.03, 0.05);
\coordinate (5) at (-1.89, -3.49);
\coordinate (6) at (0, 0);
\coordinate (7) at (-2.03, -1.39);
\coordinate (8) at (-2.32, -1.87);
\coordinate (9) at (-1.11, -1.13);
\coordinate (10) at (-2.29, -0.49);
\node at (0) [above] {\mylabel{0}};
\node at (6) [below right] {\mylabel{6}};

\begin{scope}
 \clip (6)--(5)--(3)--cycle;
 \draw[frame, face, fill=mycolor6] (2)--(8)--(0)--cycle;
 \draw[frame, face, fill=mycolor8] (5)--(0)--(8)--cycle;
 \draw[thickedge] (0) -- (8);
\end{scope}

\draw[frame, face, fill=mycolor16] (1)--(9)--(0)--cycle;
\draw[frame, face, fill=mycolor14] (6)--(1)--(9)--cycle;
\draw[thickedge] (1) -- (9);
\node at (2) [below right]{\mylabel{2}};
\node at (5) [below left] {\mylabel{5}};
\node at (1) [below left] {\mylabel{1}};
\node at (4) [above]{\mylabel{4}};
\draw[frame, face, fill=mycolor12] (2)--(10)--(1)--cycle;
\draw[frame, face, fill=mycolor10] (4)--(2)--(10)--cycle;
\draw[thickedge] (2) -- (10);

\begin{scope}
 \clip (5)--(3)--(1)--cycle;
 \draw[frame, face, fill=mycolor6] (2)--(8)--(0)--cycle;
 \draw[frame, face, fill=mycolor8] (5)--(0)--(8)--cycle;
 \draw[thickedge] (0) -- (8);
\end{scope}

\draw[frame, face, fill=mycolor7] (2)--(4)--(8)--cycle;
\draw[frame, face, fill=mycolor9] (5)--(8)--(4)--cycle;
\draw[frame, face, fill=mycolor13] (10)--(1)--(6)--cycle;
\draw[frame, face, fill=mycolor11] (4)--(10)--(6)--cycle;
\draw[frame, face, fill=mycolor17] (9)--(0)--(5)--cycle;
\draw[frame, face, fill=mycolor15] (6)--(9)--(5)--cycle;
\draw[thickedge] (4) -- (8);
\draw[thickedge] (6) -- (10);
\draw[thickedge] (5) -- (9);
\node at (10) {\mylabel{10}};
\node at (9) {\mylabel{9}};
\node at (8) {\mylabel{8}};

\end{tikzpicture}
\caption{$S/{\sim}$}
\label{5}

\end{figure}

\section{The Barnette sphere}
Note that, by adding a vertex in the interior of the upper tetrahedron of the 
triangular bipyramid of Figure \ref{Fbip}, 
we obtain a simplicial decomposition of the bipyramid as in Figure \ref{barnette-sphere} 
and by identifying the vertices $3$ and $7$ (more precisely, by identifying 
faces $\face{0,1,3}$ and $\face{0,1,7}$, $\face{1,2,3}$ and $\face{1,2,7}$, $\face{2,0,3}$ and $\face{2,0,7}$), 
we obtain a simplicial decomposition of the $3$-sphere 
which is known as the Barnette sphere \cite{barnette}. 
We denote the Barnette sphere by $S_B^3$. 

We give a simplicial decomposition of Morton's prism $P$ as follows.
Add the barycenters of the bottom and top faces,
the mid points of vertical edges and a point in the interior of $P$ as 
vertices. 
As in the previous section, we slightly twist the top of the prism and 
stretch the bottom to get a 
polyhedron in Figure \ref{trimorton}.
We triangulate it as in the figure.
We denote this simplicial complex by the same symbol $P$. 

The vertices and facets of $P$ and $S_B^3$ are listed in the following table.
\begin{table}[ht]
 \renewcommand*{\arraystretch}{1.2}
\centering
 \begin{tabular}{c|l|l}
 & $P$ & $S_B^3$\\ \hline 
 vertices & $\{0,1,2,3,4,5,6,7,8,9,10,11\}$ & $\{0,1,2,3,8,9,10,11\}$ \\ \hline 
 facets 
 & $\face{4,5,7,9}$, $\face{5,6,7,10}$, $\face{4,6,7,8}$,		     & $\face{0,1,3,9}$, $\face{1,2,3,10}$, $\face{0,2,3,8}$, 	\\ 
 & $\face{4,7,8,9}$, $\face{5,7,9,10}$, $\face{6,7,8,10}$,		     & $\face{0,3,8,9}$, $\face{1,3,9,10}$, $\face{2,3,8,10}$, 	\\ 
 & $\face{4,1,8,9}$, $\face{5,2,9,10}$, $\face{0,6,8,10}$,		     & $\face{0,1,8,9}$, $\face{1,2,9,10}$, $\face{0,2,8,10}$, 	\\ 
 & $\face{0,8,10,11}$, $\face{1,8,9,11}$, $\face{2,9,10,11}$,     & $\face{0,8,10,11}$, $\face{1,8,9,11}$, $\face{2,9,10,11}$,\\ 
 & $\face{0,1,8,11}$, $\face{1, 2,9,11}$, $\face{0,2,10,11}$,	     & $\face{0,1,8,11}$, $\face{1,2,9,11}$, $\face{0,2,10,11}$,	\\ 
& $\face{7,8,9,10}$, $\face{8,9,10,11}$, $\face{0,1,2,11}$,                & $\face{3,8,9,10}$, $\face{8,9,10,11}$, $\face{0,1,2,11}$, \\ 
 & $\face{0,1,2,3}$					     & $\face{0,1,2,3}$				   
 \end{tabular}
\end{table}

Clearly, this simplicial decomposition of $P$ induces a 
simplicial decomposition of $P/{\sim}$. 

We define a map 
\[
  q\colon \{0,1,2,3,4,5,6,7,8,9,10,11\} \to \{0,1,2,3,8,9,10,11\}
\]
by 
\[
 q(i) = \begin{cases}
	 i-4, & 4\leq i \leq 7 \\
	 i, & \text{otherwise}.
	\end{cases}
\]
The map $q$ defines a simplicial map $q\colon P \to S_B^3$, and
it induces an isomorphism of simplicial complexes:
\[
 \bar{q}\colon {P/{\sim}} \xrightarrow{\cong} S_B^3
\]
Moreover, 
under this simplicial decomposition, 
$S/{\sim}$ is the 
$2$-dimensional subcomplex in Figure \ref{barnette_mobius} 
which is a full and a half twisted M\"{o}bius band, 
and $D/{\sim}$ is its boundary.
\begin{table}[ht]
 \renewcommand*{\arraystretch}{1.2}
 \centering
 \begin{tabular}{c|l|l}
  & $S/{\sim}$ & $D/{\sim}$ \\ \hline
  vertices & $\{0,1,2,8,9,10\}$ & $\{0,1,2,8,9,10\}$ \\ \hline
  facets & $\face{0,1,8}$, $\face{0,2,8}$, & $\face{1,8}$, $\face{2,8} $, \\ 
  & $\face{0,1,9}$, $\face{1,2,9}$, & $\face{0,9}$, $\face{2,9} $, \\ 
  & $\face{0,2,10}$, $\face{1,2,10}$ & $\face{0,10}$, $\face{1,10} $ 
 \end{tabular}
\end{table}

Thus, we obtained the following:
\begin{thm}
 $\spotc$ is homeomorphic to $S^3$. 

 $\spotc[2]$ is homeomorphic to a M\"{o}bius band, and 
 $\spotc[1]$ is its boundary. 
 $\spotc[2]$ is included in $\spotc \cong S^3$ as a 
 full and a half twisted M\"{o}bius band.

 In particular, the inclusion $S^1=\spotc[1]\subset \spotc \cong S^3$ is a trefoil knot.
\end{thm}

\begin{figure}[ht]
\tikzset{face/.append style={fill opacity=0.5}}
\begin{minipage}[b]{.5\hsize}
 \centering
 \begin{tikzpicture}[scale=1.7]
 \newcommand*{\mylabel}[1]{%
 \ifcase #1 0 \or 1 \or 2 \or 3 \or 5(=1) \or 6(=2) \or 4(=0) \or 7(=3) \or 10 \or 8 \or 9 \or 11 \or \fi}
 \definecolor{mycolor0}{RGB}{255,153,153}
 \definecolor{mycolor1}{RGB}{255,168,153}
 \definecolor{mycolor2}{RGB}{255,183,153}
 \definecolor{mycolor3}{RGB}{255,168,153}
 \definecolor{mycolor4}{RGB}{255,183,153}
 \definecolor{mycolor5}{RGB}{255,153,153}
 \definecolor{mycolor6}{RGB}{255,214,153}
 \definecolor{mycolor8}{RGB}{255,214,153}
 \definecolor{mycolor7}{RGB}{255,229,153}
 \definecolor{mycolor9}{RGB}{255,229,153}
 \definecolor{mycolor10}{RGB}{250,255,153}
 \definecolor{mycolor12}{RGB}{250,255,153}
 \definecolor{mycolor11}{RGB}{235,255,153}
 \definecolor{mycolor13}{RGB}{235,255,153}
 \definecolor{mycolor14}{RGB}{204,255,153}
 \definecolor{mycolor16}{RGB}{204,255,153}
 \definecolor{mycolor15}{RGB}{189,255,153}
 \definecolor{mycolor17}{RGB}{189,255,153}
 \definecolor{mycolor18}{RGB}{173,255,153}
 \definecolor{mycolor19}{RGB}{158,255,153}
 \definecolor{mycolor20}{RGB}{153,255,163}
 \definecolor{mycolor21}{RGB}{153,255,178}
 \definecolor{mycolor22}{RGB}{153,255,193}
 \definecolor{mycolor23}{RGB}{153,255,209}
 \definecolor{mycolor24}{RGB}{153,255,224}
 \definecolor{mycolor25}{RGB}{153,255,239}
 \definecolor{mycolor26}{RGB}{153,255,255}
 \definecolor{mycolor27}{RGB}{153,239,255}
 \definecolor{mycolor28}{RGB}{153,224,255}
 \definecolor{mycolor29}{RGB}{153,209,255}
 \definecolor{mycolor30}{RGB}{153,224,255}
 \definecolor{mycolor31}{RGB}{153,209,255}
 \definecolor{mycolor32}{RGB}{153,193,255}
 \definecolor{mycolor33}{RGB}{153,178,255}
 \definecolor{mycolor34}{RGB}{153,163,255}
 \definecolor{mycolor35}{RGB}{158,153,255}
 \definecolor{mycolor36}{RGB}{173,153,255}
 \definecolor{mycolor37}{RGB}{189,153,255}
 \definecolor{mycolor38}{RGB}{204,153,255}
 \definecolor{mycolor39}{RGB}{219,153,255}
 \definecolor{mycolor40}{RGB}{234,153,255}
 \definecolor{mycolor41}{RGB}{250,153,255}
 \definecolor{mycolor42}{RGB}{255,153,245}
 \definecolor{mycolor43}{RGB}{255,153,229}
 \definecolor{mycolor44}{RGB}{255,153,214}
 \definecolor{mycolor45}{RGB}{255,153,199}
 \definecolor{mycolor46}{RGB}{255,153,183}
 \coordinate (0) at (0, 0);
 \coordinate (1) at (1.8, 0.91);
 \coordinate (2) at (-0.33, 1.24);
 \coordinate (3) at (0.4, -0.45);
 \coordinate (4) at (-0.06, 2.38);
 \coordinate (5) at (2.28, 2.78);
 \coordinate (6) at (-0.35, 2.93);
 \coordinate (7) at (0.63, 2.75);
 \coordinate (8) at (-0.03, 1.01);
 \coordinate (9) at (2.01, 1.73);
 \coordinate (10) at (-0.34, 2);
 \coordinate (11) at (0.52, 1.05);
 \newcommand*{\myentry}[1]{%
 \ifcase #1 
 \draw[frame, face, fill=mycolor0] (0)--(1)--(3)--cycle \or 
 \draw[frame, face, fill=mycolor1] (1)--(2)--(3)--cycle \or 
 \draw[frame, face, fill=mycolor2] (2)--(0)--(3)--cycle \or 
 \draw[frame, face, fill=mycolor3] (4)--(5)--(7)--cycle \or 
 \draw[frame, face, fill=mycolor4] (5)--(6)--(7)--cycle \or 
 \draw[frame, face, fill=mycolor5] (6)--(4)--(7)--cycle \or 
 \draw[frame, face, fill=mycolor6] (2)--(8)--(0)--cycle \or 
 \draw[frame, face, fill=mycolor7] (2)--(4)--(8)--cycle \or 
 \draw[frame, face, fill=mycolor8] (5)--(0)--(8)--cycle \or 
 \draw[frame, face, fill=mycolor9] (5)--(8)--(4)--cycle \or 
 \draw[frame, face, fill=mycolor10] (4)--(2)--(10)--cycle \or 
 \draw[frame, face, fill=mycolor11] (4)--(10)--(6)--cycle \or 
 \draw[frame, face, fill=mycolor12] (2)--(10)--(1)--cycle \or 
 \draw[frame, face, fill=mycolor13] (10)--(1)--(6)--cycle \or 
 \draw[frame, face, fill=mycolor14] (6)--(1)--(9)--cycle \or 
 \draw[frame, face, fill=mycolor15] (6)--(9)--(5)--cycle \or 
 \draw[frame, face, fill=mycolor16] (1)--(9)--(0)--cycle \or 
 \draw[frame, face, fill=mycolor17] (9)--(0)--(5)--cycle \or 
 \draw[frame, face, fill=mycolor18] (8)--(9)--(10)--cycle \or 
 \draw[frame, face, fill=mycolor19] (7)--(8)--(9)--cycle \or 
 \draw[frame, face, fill=mycolor20] (7)--(9)--(10)--cycle \or 
 \draw[frame, face, fill=mycolor21] (7)--(10)--(8)--cycle \or 
 \draw[frame, face, fill=mycolor22] (5)--(7)--(9)--cycle \or 
 \draw[frame, face, fill=mycolor23] (6)--(7)--(9)--cycle \or 
 \draw[frame, face, fill=mycolor24] (4)--(7)--(10)--cycle \or 
 \draw[frame, face, fill=mycolor25] (6)--(7)--(10)--cycle \or 
 \draw[frame, face, fill=mycolor26] (4)--(7)--(8)--cycle \or 
 \draw[frame, face, fill=mycolor27] (5)--(7)--(8)--cycle \or 
 \draw[frame, face, fill=mycolor28] (4)--(10)--(8)--cycle \or 
 \draw[frame, face, fill=mycolor29] (5)--(8)--(9)--cycle \or 
 \draw[frame, face, fill=mycolor30] (6)--(9)--(10)--cycle \or 
 \draw[frame, face, fill=mycolor31] (0)--(8)--(9)--cycle \or 
 \draw[frame, face, fill=mycolor32] (1)--(9)--(10)--cycle \or 
 \draw[frame, face, fill=mycolor33] (2)--(10)--(8)--cycle \or 
 \draw[frame, face, fill=mycolor34] (8)--(9)--(11)--cycle \or 
 \draw[frame, face, fill=mycolor35] (9)--(10)--(11)--cycle \or 
 \draw[frame, face, fill=mycolor36] (10)--(8)--(11)--cycle \or 
 \draw[frame, face, fill=mycolor37] (11)--(0)--(9)--cycle \or 
 \draw[frame, face, fill=mycolor38] (11)--(1)--(9)--cycle \or 
 \draw[frame, face, fill=mycolor39] (11)--(1)--(10)--cycle \or 
 \draw[frame, face, fill=mycolor40] (11)--(2)--(10)--cycle \or 
 \draw[frame, face, fill=mycolor41] (11)--(2)--(8)--cycle \or 
 \draw[frame, face, fill=mycolor42] (11)--(0)--(8)--cycle \or 
 \draw[frame, face, fill=mycolor43] (11)--(0)--(1)--cycle \or 
 \draw[frame, face, fill=mycolor44] (11)--(1)--(2)--cycle \or 
 \draw[frame, face, fill=mycolor45] (11)--(2)--(0)--cycle \or 
 \draw[frame, face, fill=mycolor46] (0)--(1)--(2)--cycle \or 
 \draw[thickedge] (1) -- (9) \or 
 \draw[thickedge] (5) -- (9) \or 
 \draw[thickedge] (2) -- (10) \or 
 \draw[thickedge] (6) -- (10) \or 
 \draw[thickedge] (0) -- (8) \or 
 \draw[thickedge] (4) -- (8) \or 
 \node at (0) [left] {\mylabel{0}} \or 
 \node at (1) [right] {\mylabel{1}} \or 
 \node at (2) [left] {\mylabel{2}} \or 
 \node at (3) [below] {\mylabel{3}} \or 
 \node at (4) [below right]{\mylabel{4}} \or 
 \node at (5) [right] {\mylabel{5}} \or 
 \node at (6) [left] {\mylabel{6}} \or 
 \node at (7) [above=2pt] {\mylabel{7}} \or 
 \node at (8) [below right] {\mylabel{8}} \or 
 \node at (9) [right] {\mylabel{9}} \or 
 \node at (10) [left] {\mylabel{10}} \or 
 \node at (11) [below right] {\mylabel{11}} \or 
 \fi}
 \foreach \x in {
 1  
 ,12  
 ,13  
 ,15  
 ,14  
 ,2  
 ,46  
 ,32  
 ,30  
 ,44  
 ,45  
 ,43  
 ,40  
 ,39  
 ,38  
 ,41  
 ,42  
 ,37  
 ,23  
 ,22  
 ,25  
 ,35  
 ,36  
 ,34  
 ,33  
 ,18  
 ,20  
 ,21  
 ,19  
 ,24  
 ,28  
 ,26  
 ,29  
 ,31  
 ,27  
 ,64  
 ,10  
 ,11  
 ,6  
 ,7  
 ,4  
 ,5  
 ,3  
 ,9  
 ,8  
 ,17  
 ,16  
 ,0  
 ,50  
 ,49  
 ,47  
 ,48  
 ,51  
 ,52  
 ,62  
 ,59  
 ,56  
 ,55  
 ,54  
 ,63  
 ,60  
 ,58  
 ,53  
 ,61  
 ,57  
 }
 {
 \myentry{\x};
 }
 \end{tikzpicture}
\caption{$P$ (and $P/{\sim}$)}
\label{trimorton}
\end{minipage}
\begin{minipage}[b]{.46\hsize}
\centering
\begin{tikzpicture}[scale=1.3]
\newcommand*{\mylabel}[1]{%
\ifcase #1 0 \or 1 \or 2 \or 3 \or {7=3} \or 10 \or 8 \or 9 \or 11 \or \fi}
\definecolor{mycolor0}{RGB}{255,153,153}
\definecolor{mycolor1}{RGB}{255,168,153}
\definecolor{mycolor2}{RGB}{255,183,153}
\definecolor{mycolor3}{RGB}{255,168,153}
\definecolor{mycolor4}{RGB}{255,183,153}
\definecolor{mycolor5}{RGB}{255,153,153}
\definecolor{mycolor6}{RGB}{255,214,153}
\definecolor{mycolor7}{RGB}{255,229,153}
\definecolor{mycolor8}{RGB}{250,255,153}
\definecolor{mycolor9}{RGB}{235,255,153}
\definecolor{mycolor10}{RGB}{204,255,153}
\definecolor{mycolor11}{RGB}{189,255,153}
\definecolor{mycolor12}{RGB}{173,255,153}
\definecolor{mycolor13}{RGB}{158,255,153}
\definecolor{mycolor14}{RGB}{153,255,163}
\definecolor{mycolor15}{RGB}{153,255,178}
\definecolor{mycolor16}{RGB}{153,255,193}
\definecolor{mycolor17}{RGB}{153,255,209}
\definecolor{mycolor18}{RGB}{153,255,224}
\definecolor{mycolor19}{RGB}{153,255,239}
\definecolor{mycolor20}{RGB}{153,255,255}
\definecolor{mycolor21}{RGB}{153,239,255}
\definecolor{mycolor22}{RGB}{153,224,255}
\definecolor{mycolor23}{RGB}{153,209,255}
\definecolor{mycolor24}{RGB}{153,224,255}
\definecolor{mycolor25}{RGB}{153,209,255}
\definecolor{mycolor26}{RGB}{153,193,255}
\definecolor{mycolor27}{RGB}{153,178,255}
\definecolor{mycolor28}{RGB}{153,163,255}
\definecolor{mycolor29}{RGB}{158,153,255}
\definecolor{mycolor30}{RGB}{173,153,255}
\definecolor{mycolor31}{RGB}{189,153,255}
\definecolor{mycolor32}{RGB}{204,153,255}
\definecolor{mycolor33}{RGB}{219,153,255}
\definecolor{mycolor34}{RGB}{234,153,255}
\definecolor{mycolor35}{RGB}{250,153,255}
\definecolor{mycolor36}{RGB}{255,153,245}
\definecolor{mycolor37}{RGB}{255,153,229}
\definecolor{mycolor38}{RGB}{255,153,214}
\definecolor{mycolor39}{RGB}{255,153,199}
\definecolor{mycolor40}{RGB}{255,153,183}
\coordinate (0) at (0, 0);
\coordinate (1) at (2.67, 1.34);
\coordinate (2) at (-0.5, 1.83);
\coordinate (3) at (0.6, -0.68);
\coordinate (4) at (0.93, 4.07);
\coordinate (5) at (0.64, 2.25);
\coordinate (6) at (0.3, 1.85);
\coordinate (7) at (1.43, 1.93);
\coordinate (8) at (0.77, 1.78);
\newcommand*{\myentry}[1]{%
\ifcase #1 
\draw[frame, face, fill=mycolor0] (0)--(1)--(3)--cycle \or 
\draw[frame, face, fill=mycolor1] (1)--(2)--(3)--cycle \or 
\draw[frame, face, fill=mycolor2] (0)--(2)--(3)--cycle \or 
\draw[frame, face, fill=mycolor3] (1)--(2)--(4)--cycle \or 
\draw[frame, face, fill=mycolor4] (0)--(2)--(4)--cycle \or 
\draw[frame, face, fill=mycolor5] (0)--(1)--(4)--cycle \or 
\draw[frame, face, fill=mycolor6] (0)--(2)--(5)--cycle \or 
\draw[frame, face, fill=mycolor7] (1)--(2)--(5)--cycle \or 
\draw[frame, face, fill=mycolor8] (1)--(2)--(7)--cycle \or 
\draw[frame, face, fill=mycolor9] (0)--(1)--(7)--cycle \or 
\draw[frame, face, fill=mycolor10] (0)--(1)--(6)--cycle \or 
\draw[frame, face, fill=mycolor11] (0)--(2)--(6)--cycle \or 
\draw[frame, face, fill=mycolor12] (5)--(6)--(7)--cycle \or 
\draw[frame, face, fill=mycolor13] (4)--(5)--(6)--cycle \or 
\draw[frame, face, fill=mycolor14] (4)--(6)--(7)--cycle \or 
\draw[frame, face, fill=mycolor15] (4)--(5)--(7)--cycle \or 
\draw[frame, face, fill=mycolor16] (2)--(4)--(6)--cycle \or 
\draw[frame, face, fill=mycolor17] (0)--(4)--(6)--cycle \or 
\draw[frame, face, fill=mycolor18] (1)--(4)--(7)--cycle \or 
\draw[frame, face, fill=mycolor19] (0)--(4)--(7)--cycle \or 
\draw[frame, face, fill=mycolor20] (1)--(4)--(5)--cycle \or 
\draw[frame, face, fill=mycolor21] (2)--(4)--(5)--cycle \or 
\draw[frame, face, fill=mycolor22] (1)--(5)--(7)--cycle \or 
\draw[frame, face, fill=mycolor23] (2)--(5)--(6)--cycle \or 
\draw[frame, face, fill=mycolor24] (0)--(6)--(7)--cycle \or 
\draw[frame, face, fill=mycolor25] (0)--(5)--(6)--cycle \or 
\draw[frame, face, fill=mycolor26] (1)--(6)--(7)--cycle \or 
\draw[frame, face, fill=mycolor27] (2)--(5)--(7)--cycle \or 
\draw[frame, face, fill=mycolor28] (5)--(6)--(8)--cycle \or 
\draw[frame, face, fill=mycolor29] (6)--(7)--(8)--cycle \or 
\draw[frame, face, fill=mycolor30] (5)--(7)--(8)--cycle \or 
\draw[frame, face, fill=mycolor31] (0)--(6)--(8)--cycle \or 
\draw[frame, face, fill=mycolor32] (1)--(6)--(8)--cycle \or 
\draw[frame, face, fill=mycolor33] (1)--(7)--(8)--cycle \or 
\draw[frame, face, fill=mycolor34] (2)--(7)--(8)--cycle \or 
\draw[frame, face, fill=mycolor35] (2)--(5)--(8)--cycle \or 
\draw[frame, face, fill=mycolor36] (0)--(5)--(8)--cycle \or 
\draw[frame, face, fill=mycolor37] (0)--(1)--(8)--cycle \or 
\draw[frame, face, fill=mycolor38] (1)--(2)--(8)--cycle \or 
\draw[frame, face, fill=mycolor39] (0)--(2)--(8)--cycle \or 
\draw[frame, face, fill=mycolor40] (0)--(1)--(2)--cycle \or 
\draw[thickedge] (1) -- (6) \or 
\draw[thickedge] (2) -- (6) \or 
\draw[thickedge] (2) -- (7) \or 
\draw[thickedge] (0) -- (7) \or 
\draw[thickedge] (0) -- (5) \or 
\draw[thickedge] (1) -- (5) \or 
\node at (0) [left] {\mylabel{0}} \or 
\node at (1) [right] {\mylabel{1}} \or 
\node at (2) [left] {\mylabel{2}} \or 
\node at (3) [below] {\mylabel{3}} \or 
\node at (4) [above] {\mylabel{4}} \or 
\node at (5) [right] {\mylabel{5}} \or 
\node at (6) [above left] {\mylabel{6}} \or 
\node at (7) [above right] {\mylabel{7}} \or 
\node at (8) [below right] {\mylabel{8}} \or 
\fi}
\foreach \x in {
50  
,49  
,1  
,3  
,2  
,40  
,7  
,21  
,20  
,46  
,22  
,8  
,27  
,43  
,38  
,34  
,39  
,35  
,6  
,23  
,37  
,36  
,25  
,45  
,13  
,15  
,30  
,28  
,29  
,33  
,31  
,32  
,26  
,10  
,41  
,12  
,14  
,24  
,18  
,11  
,9  
,16  
,42  
,17  
,19  
,44  
,52  
,55  
,54  
,53  
,0  
,4  
,5  
,48  
,47  
,51  
}
{
\myentry{\x};
}
\end{tikzpicture}
\caption{$S_B^3$}
\label{barnette-sphere}
\end{minipage}
\end{figure}

\begin{figure}[ht]
\centering
\begin{tikzpicture}[scale=1.4]
\newcommand*{\mylabel}[1]{%
\ifcase #1 0 \or 1 \or 2 \or 3 \or {7=3} \or 10 \or 8 \or 9 \or 11 \or \fi}
\definecolor{mycolor0}{RGB}{255,153,153}
\definecolor{mycolor1}{RGB}{255,168,153}
\definecolor{mycolor2}{RGB}{255,183,153}
\definecolor{mycolor3}{RGB}{255,168,153}
\definecolor{mycolor4}{RGB}{255,183,153}
\definecolor{mycolor5}{RGB}{255,153,153}
\definecolor{mycolor6}{RGB}{255,214,153}
\definecolor{mycolor7}{RGB}{255,229,153}
\definecolor{mycolor8}{RGB}{250,255,153}
\definecolor{mycolor9}{RGB}{235,255,153}
\definecolor{mycolor10}{RGB}{204,255,153}
\definecolor{mycolor11}{RGB}{189,255,153}
\definecolor{mycolor12}{RGB}{173,255,153}
\definecolor{mycolor13}{RGB}{158,255,153}
\definecolor{mycolor14}{RGB}{153,255,163}
\definecolor{mycolor15}{RGB}{153,255,178}
\definecolor{mycolor16}{RGB}{153,255,193}
\definecolor{mycolor17}{RGB}{153,255,209}
\definecolor{mycolor18}{RGB}{153,255,224}
\definecolor{mycolor19}{RGB}{153,255,239}
\definecolor{mycolor20}{RGB}{153,255,255}
\definecolor{mycolor21}{RGB}{153,239,255}
\definecolor{mycolor22}{RGB}{153,224,255}
\definecolor{mycolor23}{RGB}{153,209,255}
\definecolor{mycolor24}{RGB}{153,224,255}
\definecolor{mycolor25}{RGB}{153,209,255}
\definecolor{mycolor26}{RGB}{153,193,255}
\definecolor{mycolor27}{RGB}{153,178,255}
\definecolor{mycolor28}{RGB}{153,163,255}
\definecolor{mycolor29}{RGB}{158,153,255}
\definecolor{mycolor30}{RGB}{173,153,255}
\definecolor{mycolor31}{RGB}{189,153,255}
\definecolor{mycolor32}{RGB}{204,153,255}
\definecolor{mycolor33}{RGB}{219,153,255}
\definecolor{mycolor34}{RGB}{234,153,255}
\definecolor{mycolor35}{RGB}{250,153,255}
\definecolor{mycolor36}{RGB}{255,153,245}
\definecolor{mycolor37}{RGB}{255,153,229}
\definecolor{mycolor38}{RGB}{255,153,214}
\definecolor{mycolor39}{RGB}{255,153,199}
\definecolor{mycolor40}{RGB}{255,153,183}
\coordinate (0) at (0, 0);
\coordinate (1) at (2.67, 1.34);
\coordinate (2) at (-0.5, 1.83);
\coordinate (3) at (0.6, -0.68);
\coordinate (4) at (0.93, 4.07);
\coordinate (5) at (0.64, 2.25);
\coordinate (6) at (0.3, 1.85);
\coordinate (7) at (1.43, 1.99);
\coordinate (8) at (0.77, 1.78);
\newcommand*{\myentry}[1]{%
\ifcase #1 
\draw[frame, face, fill=mycolor0] (0)--(1)--(3)--cycle \or 
\draw[frame, face, fill=mycolor1] (1)--(2)--(3)--cycle \or 
\draw[frame, face, fill=mycolor2] (0)--(2)--(3)--cycle \or 
\draw[frame, face, fill=mycolor3] (1)--(2)--(4)--cycle \or 
\draw[frame, face, fill=mycolor4] (0)--(2)--(4)--cycle \or 
\draw[frame, face, fill=mycolor5] (0)--(1)--(4)--cycle \or 
\draw[frame, face, fill=mycolor6] (0)--(2)--(5)--cycle \or 
\draw[frame, face, fill=mycolor7] (1)--(2)--(5)--cycle \or 
\draw[frame, face, fill=mycolor8] (1)--(2)--(7)--cycle \or 
\draw[frame, face, fill=mycolor9] (0)--(1)--(7)--cycle \or 
\draw[frame, face, fill=mycolor10] (0)--(1)--(6)--cycle \or 
\draw[frame, face, fill=mycolor11] (0)--(2)--(6)--cycle \or 
\draw[frame, face, fill=mycolor12] (5)--(6)--(7)--cycle \or 
\draw[frame, face, fill=mycolor13] (4)--(5)--(6)--cycle \or 
\draw[frame, face, fill=mycolor14] (4)--(6)--(7)--cycle \or 
\draw[frame, face, fill=mycolor15] (4)--(5)--(7)--cycle \or 
\draw[frame, face, fill=mycolor16] (2)--(4)--(6)--cycle \or 
\draw[frame, face, fill=mycolor17] (0)--(4)--(6)--cycle \or 
\draw[frame, face, fill=mycolor18] (1)--(4)--(7)--cycle \or 
\draw[frame, face, fill=mycolor19] (0)--(4)--(7)--cycle \or 
\draw[frame, face, fill=mycolor20] (1)--(4)--(5)--cycle \or 
\draw[frame, face, fill=mycolor21] (2)--(4)--(5)--cycle \or 
\draw[frame, face, fill=mycolor22] (1)--(5)--(7)--cycle \or 
\draw[frame, face, fill=mycolor23] (2)--(5)--(6)--cycle \or 
\draw[frame, face, fill=mycolor24] (0)--(6)--(7)--cycle \or 
\draw[frame, face, fill=mycolor25] (0)--(5)--(6)--cycle \or 
\draw[frame, face, fill=mycolor26] (1)--(6)--(7)--cycle \or 
\draw[frame, face, fill=mycolor27] (2)--(5)--(7)--cycle \or 
\draw[frame, face, fill=mycolor28] (5)--(6)--(8)--cycle \or 
\draw[frame, face, fill=mycolor29] (6)--(7)--(8)--cycle \or 
\draw[frame, face, fill=mycolor30] (5)--(7)--(8)--cycle \or 
\draw[frame, face, fill=mycolor31] (0)--(6)--(8)--cycle \or 
\draw[frame, face, fill=mycolor32] (1)--(6)--(8)--cycle \or 
\draw[frame, face, fill=mycolor33] (1)--(7)--(8)--cycle \or 
\draw[frame, face, fill=mycolor34] (2)--(7)--(8)--cycle \or 
\draw[frame, face, fill=mycolor35] (2)--(5)--(8)--cycle \or 
\draw[frame, face, fill=mycolor36] (0)--(5)--(8)--cycle \or 
\draw[frame, face, fill=mycolor37] (0)--(1)--(8)--cycle \or 
\draw[frame, face, fill=mycolor38] (1)--(2)--(8)--cycle \or 
\draw[frame, face, fill=mycolor39] (0)--(2)--(8)--cycle \or 
\draw[frame, face, fill=mycolor40] (0)--(1)--(2)--cycle \or 
\draw[thickedge] (1) -- (6) \or 
\draw[thickedge] (2) -- (6) \or 
\draw[thickedge] (2) -- (7) \or 
\draw[thickedge] (0) -- (7) \or 
\draw[thickedge] (0) -- (5) \or 
\draw[thickedge] (1) -- (5) \or 
\node at (0) [left] {\mylabel{0}} \or 
\node at (1) [right] {\mylabel{1}} \or 
\node at (2) [left] {\mylabel{2}} \or 
\node at (3) [below] {\mylabel{3}} \or 
\node at (4) [above] {\mylabel{4}} \or 
\node at (5) [right] {\mylabel{5}} \or 
\node at (6) [above left] {\mylabel{6}} \or 
\node at (7) [above right] {\mylabel{7}} \or 
\node at (8) [below right] {\mylabel{8}} \or 
\fi}
\foreach \x in {
49  
,7  
,46  
,8  
,43  
,6  
,45  
,10  
,41  
,11  
,9  
,42  
,44  
,52  
,54  
,53  
,48  
,47  
}
{
\myentry{\x};
}
\end{tikzpicture}
\caption{$S/{\sim}$ and $D/{\sim}$}
\label{barnette_mobius}
\end{figure}
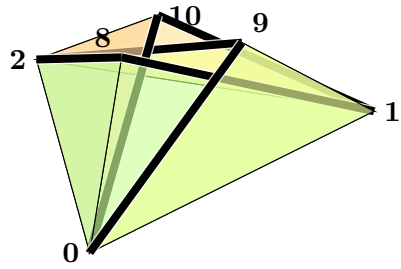

\noindent\textit{Department of Mathematical Sciences, 
University of the Ryukyus\\
Nishihara-cho, Okinawa 903-0213, 
Japan}

\end{document}